\documentclass[review]{elsarticle}

\usepackage{bm,amssymb,amsmath,mathrsfs}
\usepackage{amsthm}
\usepackage{lineno}
\usepackage[usenames]{color}

\usepackage{dcolumn}

\newcommand{\bq}{\begin{equation}}
\newcommand{\eq}{\end{equation}}
\newcommand{\bc}{\begin{center}}
\newcommand{\ec}{\end{center}}
\newcommand{\bit}{\begin{itemize}}
\newcommand{\eit}{\end{itemize}}
\newcommand{\ben}{\begin{enumerate}}
\newcommand{\een}{\end{enumerate}}

\theoremstyle{plain}
\newtheorem{theorem}{Theorem}[section]
\newtheorem*{theorem*}{Theorem}
\newtheorem{proposition}[theorem]{Proposition}
\newtheorem{lemma}[theorem]{Lemma}

\newtheorem{remark}[theorem]{Remark}
\newtheorem{definition}[theorem]{Definition}
\newtheorem{conjecture}[theorem]{Conjecture}

\begin{document}

\journal{(internal report CC23-5)}

\begin{frontmatter}

\title{First double mode of the Poisson distribution of order $k$}

\author[cc]{S.~R.~Mane}
\ead{srmane001@gmail.com}
\address[cc]{Convergent Computing Inc., P.~O.~Box 561, Shoreham, NY 11786, USA}

\begin{abstract}
The Poisson distribution of order $k$ is a special case of a compound Poisson distribution.
For $k=1$ it is the standard Poisson distribution.
Our focus in this note is for $k\ge2$.
For sufficiently small values of the rate parameter $\lambda$, both the median and mode equal zero.
The median is zero if and only if $\lambda \le (\ln2)/k$.
The supremum value of $\lambda$ for the mode to be zero is known only for small values of $k$.
This note presents results for the ``first double mode''
by which is meant the first occasion (smallest value of $\lambda$) the distribution is bimodal, with modes at $0$ and $m>0$.
Next, an almost complete answer is supplied to the question ``which positive integers cannot be modes of the Poisson distribution of order $k$?''
The term ``almost complete'' signifies that some parts of the answer are conjectures based on numerical searches.
However, if the conjectures are proved to be correct, the solution presented in this note is complete: all parameter values are covered.
\end{abstract}

\vskip 0.25in


\end{frontmatter}

\newpage
\setcounter{equation}{0}
\section{\label{sec:intro} Introduction}
In a recent note \cite{Mane_Poisson_k_CC23_3},
the author presented numerical solutions for asymptotic results for the Poisson distribution of order $k$ \cite{PhilippouGeorghiouPhilippou}.
It is a variant (or extension) of the well-known Poisson distribution.
We begin with its formal definition. 
\begin{definition}
  \label{def:pmf_Poisson_order_k}
  The Poisson distribution of order $k$ (where $k\ge1$ is an integer) and parameter $\lambda > 0$
  is an integer-valued statistical distribution with the probability mass function (pmf)
\bq
\label{eq:pmf_Poisson_order_k}
f_k(n;\lambda) = e^{-k\lambda}\sum_{n_1+2n_2+\dots+kn_k=n} \frac{\lambda^{n_1+\dots+n_k}}{n_1!\dots n_k!} \,, \qquad n=0,1,2\dots
\eq
\end{definition}
\noindent
For $k=1$ it is the standard Poisson distribution.
The Poisson distribution of order $k$ is a special case of the compound Poisson distribution introduced by Adelson \cite{Adelson1966}.
Although exact expressions for the mean and variance of the Poisson distribution of order $k$ are known \cite{PhilippouMeanVar},
exact results for its median and mode are difficult to obtain.
For sufficiently small values of the rate parameter $\lambda$, both the median and mode equal zero.
It was derived in \cite{Mane_Poisson_k_CC23_3} that the median is zero if and only if $\lambda \le (\ln2)/k$.
The supremum value of $\lambda$ for the mode to be zero is known only for small values of $k$ \cite{PhilippouFibQ,KwonPhilippou}.
This note presents results for the ``first double mode''
by which is meant the first occasion (smallest value of $\lambda$) the distribution is bimodal, with modes at $0$ and $m>0$.

In addition, Kwon and Philippou \cite{KwonPhilippou} posed the following question at the end of their paper:
``What positive integers cannot be modes of the Poisson distribution of order $k$?''
This note presents an almost complete answer to the above question.
The qualifier ``almost complete'' signifies that some parts of the answer are conjectures based on numerical studies.
However, if the conjectures are proved to be correct, the solution presented in this note is complete: all parameter values are covered.

The structure of this paper is as follows.
Sec.~\ref{sec:notation} presents basic definitions and notation employed in this note.
Sec.~\ref{sec:KPLemmas} lists lemmas by Kwon and Philippou \cite{KwonPhilippou}, reproduced for ease of reference.
Sec.~\ref{sec:firstdoublemode} presents results for the double mode.
Sec.~\ref{sec:notmode} presents the ``almost complete'' answer mentioned above.
Sec.~\ref{sec:conc} concludes.

\newpage
\setcounter{equation}{0}
\section{\label{sec:notation}Basic notation and definitions}
For later reference we define the parameter $\kappa=k(k+1)/2$.
We denote the mode $m$ (with pertinent subscripts, etc.~to denote the dependence on $k$ and $\lambda$, see below).
The mode is defined as the location(s) of the {\em global maximum} of the probability mass function.
It is known that the mode may not be unique.
For the standard Poisson distribution with parameter $\lambda$, the mode equals $\lfloor\lambda\rfloor$ if $\lambda\not\in\mathbb{N}$,
but both $\lambda-1$ and $\lambda$ are modes if $\lambda\in\mathbb{N}$.
Kwon and Philippou \cite{KwonPhilippou} published a table of values of $\lambda$ where the Poisson distribution of order $k$ has a double mode,
for $2 \le k \le 4$ and $0 < \lambda \le 2$.

We study the problem of the ``first double mode''
by which is meant the first occasion (smallest value of $\lambda$) the distribution is bimodal, with modes at $0$ and $m>0$.
Philippou \cite{PhilippouFibQ} showed that for $k=2$, the answer is $\lambda=\sqrt{3}-1$.
Kwon and Philippou \cite{KwonPhilippou} published the answers for $k=3$ and $4$ (roots of cubic and quartic polynomials, respectively).
This note studies the general case $k\ge2$.
We follow the notation in \cite{KwonPhilippou} and
define $h_k(n;\lambda) = e^{k\lambda}f_k(n;\lambda)$, with initial value $h_k(0;\lambda) = 1$ (eq.~(6) in \cite{KwonPhilippou}).
Observe from eq.~\eqref{eq:pmf_Poisson_order_k} that $h_k(n;\lambda)$ is a polynomial in $\lambda$.
Note the following:
\begin{enumerate}
\item
  The polynomial $h_k(n;\lambda)$ has degree $n$ because the highest power of $\lambda$ in
  the sum in eq.~\eqref{eq:pmf_Poisson_order_k} is given by the tuple $(n_1,\dots,n_k)=(n,0,\dots,0)$.
  The resulting term in the sum is $\lambda^n/n!$.
\item
  Next, $h_k(0;\lambda)=1$ and $h_k(n;\lambda)$ has no constant term for $n>0$.
\item
  All the coefficients in $h_k(n;\lambda)$ are positive.
\item
  Hence for $n>0$, $h_k(n;0)=0$ and $h_k(n;\lambda)$ is strictly increasing in $\lambda$ for $\lambda\ge0$.
\end{enumerate}
Note also that for fixed $\lambda$ and $n\ge0$, the polynomials $h_k(n;\lambda)$ are identical for all $k\ge n$.
\begin{proposition}\label{prop:hk_same_k_ge_n}
  For fixed $\lambda$ and $n\ge0$, $h_k(n;\lambda) = h_n(n;\lambda)$ for all $k > n$.
\end{proposition}
\begin{proof}
  If $n=0$ all the polynomials are identically $1$. Hence assume $n>0$.
  Observe from eq.~\eqref{eq:pmf_Poisson_order_k} that if $k > n$ then $x_j=0$ for $j>n$.
  Hence the sums for $h_k(n;\lambda)$ and $h_n(n;\lambda)$ contain the same tuples.
\end{proof}
\noindent
Note also for fixed $\lambda>0$ and $n\ge2$, the polynomials $h_k(n;\lambda)$ form an increasing sequence for $k=1,\dots,n$.
\begin{proposition}\label{prop:hk_inc_seq_k_le_n}
  For fixed $\lambda>0$ and $n\ge2$, we have the following sequence for $k=1,\dots,n$:
\bq
\label{eq:hk_inc_seq_k_le_n}
\frac{\lambda^n}{n!} = h_1(n;\lambda) < \dots < h_n(n;\lambda) \,.
\eq
\end{proposition}
\begin{proof}
  The cases $n=0$ and $n=1$ are vacuous (no sequence), hence we must have $n\ge2$.
  Fix a value $k\le n$.
  Then the terms in the sum in eq.~\eqref{eq:pmf_Poisson_order_k} contain exactly $k$ summands.
  By conditioning on the value of $n_k$, we obtain the following sum
\bq
\label{eq:rec_poly}
  h_k(n;\lambda) = \sum_{n_k=0}^{\lfloor(n/k)\rfloor} \frac{\lambda^{n_k}}{n_k!}\,h_{k-1}(n-kn_k;\lambda) \,.
\eq
Split off the term $n_k=0$ to obtain
\bq
h_k(n;\lambda) = h_{k-1}(n;\lambda) + \sum_{n_k=1}^{\lfloor(n/k)\rfloor} \frac{\lambda^{n_k}}{n_k!}\,h_{k-1}(n-kn_k;\lambda) \,.
\eq
The sum is nonempty because $k\le n$ and its value is positive for $\lambda>0$,
hence $h_k(n;\lambda) > h_{k-1}(n;\lambda)$.
The starting value of the sequence is $h_1(n;\lambda) = \lambda^n/n!$.
\end{proof}
\begin{remark}
  Observe that eq.~\eqref{eq:rec_poly} furnishes an alternative recurrence for $h_k(n;\lambda)$,
  as a sum of lower degree polynomials $h_{k-1}(\cdot)$.
\end{remark}

\newpage
\setcounter{equation}{0}
\section{\label{sec:KPLemmas}Useful lemmas by Kwon and Philippou}
For ease of reference, we state Lemma 1 by Kwon and Philippou \cite{KwonPhilippou}, because it will be useful below.
\begin{lemma}(restatement of Lemma 1 in \cite{KwonPhilippou})
For $2 \le n \le k$ and a fixed $\lambda>0$,
\bq
\label{eq:KP_Lemma1}
\lambda \le h_k(n-1;\lambda) < h_k(n;\lambda) \,.
\eq
\end{lemma}  
\noindent
We can express eq.~\eqref{eq:KP_Lemma1} as the following increasing sequence, which is in fact more useful
\bq  
\label{eq:KP_Lemma1_inc_seq}
\lambda = h_k(1;\lambda) < h_k(2;\lambda) < \dots < h_k(k;\lambda) \,.
\eq
Kwon and Philippou \cite{KwonPhilippou} defined $r_k$ as the positive root of the equation $h_k(k,\lambda)=1$.
They proved that $r_k$ is unique and $r_k\in(0,1)$.
We supply a proof which uses only the Intermediate Value Theorem and no calculus.
\begin{lemma}(Lemma 2 in \cite{KwonPhilippou})
For $k\ge2$ and $0<\lambda<1$, the equation $h_k(k;\lambda) = h_k(0;\lambda)$ has exactly one root $\lambda=r_k\;(0<r_k<1)$ such that
\bq
\label{eq:KP_Lemma2_ineq}
  \left\{
  \begin{array}{lll}
    h_k(0;\lambda) > h_k(k;\lambda)\, &\qquad & \mathit{if}\; 0<\lambda<r_k \\
    h_k(0;\lambda) < h_k(k;\lambda)\, &\qquad & \mathit{if}\; r_k<\lambda<1\,.
  \end{array}
  \right.
\eq
\end{lemma}  
\begin{proof}
Since $h_k(0;\lambda)=1$, the equation to solve is $h_k(n;\lambda)=1$.
Also (because $k>0$) $h_k(k;\lambda)$ has no constant term hence $h_k(k;0)=0$.
Next note that the sum $n_1+2n_2+\dots+kn_k=k$ contains a tuple $(n_1,\dots,n_k)=(0,\dots,0,1)$, i.e.~$n_k=1$ and the other $n_i$ are zero.
The corresponding term in $h_k(k;\lambda)$ is simply $\lambda^1/1!$, i.e.~$\lambda$ (and it is the {\em only} term containing $n_k$).
Hence $h_k(k;\lambda) = \lambda +$(positive terms) for $\lambda>0$.
Set $\lambda=1$ to deduce $h_k(k;1)>1$.
Hence by the Intermediate Value Theorem, the equation $h_k(k;\lambda)=1$ has a root $r_k$ in the interval $r_k\in(0,1)$.
Lastly (also because $k>0$) $h_k(k;\lambda)$ is strictly increasing in $\lambda$ for $\lambda \ge 0$,
hence the root $r_k$ is unique and the inequalities in eq.~\eqref{eq:KP_Lemma2_ineq} also follow.
\end{proof}
\noindent
Also for ease of reference, we state Lemma 3 by Kwon and Philippou \cite{KwonPhilippou}, because it will be useful below.
\begin{lemma}(restatement of Lemma 3 in \cite{KwonPhilippou})
For $k \ge 2$ and $0 < \lambda \le r_k < 1$,
\bq
\label{eq:KP_Lemma3}
h_k(k;\lambda) > h_k(k+1;\lambda) \,.
\eq
\end{lemma}  
\begin{remark}  
  We can improve the lower bound on $r_k$ as follows: $r_k \in [1/\kappa,1)$.
\end{remark}  
\begin{proof}
  Philippou \cite{PhilippouFibQ} proved that the Poisson distribution of order $k$ has a unique mode of zero if $\lambda < 1/\kappa$.
  But if $\lambda=r_k$ then $h_k(k;\lambda)=1$ and zero is not a unique mode.
  Hence $r_k \ge 1/\kappa$.
\end{proof}  
\begin{remark}  
  Combining eqs.~\eqref{eq:KP_Lemma1_inc_seq} and \eqref{eq:KP_Lemma3} reveals that if $0 < \lambda \le r_k$,
  then the value of $h_k(k;\lambda)$ is a local maximum.
  This does {\em not} imply the value $n=k$ is a mode, which requires a {\em global} maximum.
\end{remark}  
\noindent
A histogram plot of $h_k(n;\lambda)$ for the Poisson distribution of order $50$ is displayed in Fig.~\ref{fig:hist50}.
Observe that $h_k(n;\lambda)$ increases for $1 \le n \le k (=50)$.
This is the increasing sequence of Lemma 1 by Kwon and Philippou \cite{KwonPhilippou} (see eq.~\eqref{eq:KP_Lemma1_inc_seq}).
Here $\lambda \simeq 0.10194$, which is small enough to satisfy the requirement of Lemma 3 in \cite{KwonPhilippou}.
The histogram indeed dips at $n=k+1=51$, in agreement with eq.~\eqref{eq:KP_Lemma3}.
Observe that $n=k=50$ is the location of a local but not a global maximum.
For values $n>k$, the structure of the histogram is complicated.
There is another local maximum at $n=98$ but its height is $0.9835\dots$ and is not a global maximum.
The above value of $\lambda$ yields the first double mode for $k=50$.
The (joint) global maximum is at $n=0$ and $113$, where the heights equal unity.

\newpage
\setcounter{equation}{0}
\section{\label{sec:firstdoublemode}First double mode}
\subsection{\label{sec:double_mode_notation}Notation for first double mode}
To avoid confusion below about the use of $m$ when discussing the mode,
let $k\ge2$ and denote the value of the first double mode by $\hat{m}_k$
and the corresponding value of $\lambda$ by $\hat\lambda_k$.
Then by definition $\hat{m}_k>0$ and
\bq
\label{eq:k_hat_m_hat_lam}
h_k(\hat{m}_k,\hat\lambda_k) = 1 \,.
\eq
We wish to prove that if $k_2>k_1$ then (i) $\hat{m}_{k_2}>\hat{m}_{k_1}$ and (ii) $\hat\lambda_{k_2}<\hat\lambda_{k_1}$.
For $k=1$ it is known that $\hat{m}_1=1$ and $\hat\lambda_1=1$.

\subsection{\label{sec:mode_ge_k}Sharp lower bound for the location of the first double mode}
\begin{proposition}
  For any $k\ge2$, if the first double modes are located at $0$ and $m$, then $m\ge k$ (which is a sharp lower bound).
\end{proposition}
\begin{proof}
  We argue by contradiction.
  Fix $k\ge2$ and suppose the first double modes are located $0$ and $\tilde{m}$, where $\tilde{m} < k$, and that $\lambda=\tilde\lambda$.
  By definition $h_k(\tilde{m};\tilde\lambda) = 1$.
  But Lemma 1 by Kwon and Philippou \cite{KwonPhilippou} then tells us that $h_k(\tilde{m}+1;\tilde\lambda) > 1$.
  {\em Hence there is a smaller value $\hat\lambda<\tilde\lambda$ such that $h_k(\tilde{m}+1;\hat\lambda) = 1$.}
  Hence $\tilde\lambda$ is {\em not} the smallest value of $\lambda$ to attain a double mode, and $\tilde{m}$ is {\em not} the location of the first double mode.
  Furthermore, we know from Philippou \cite{PhilippouFibQ} that for $k=2$,
  the locations of the first double modes are at $0$ and $2$, i.e.~$m=k$ for $k=2$ (and $\lambda=\sqrt{3}-1$).
  Hence the lower bound $m=k$ is attained, i.e.~it is a sharp lower bound.
\end{proof}

\subsection{\label{sec:bound_on_hat_lam_k}Bounds on $\hat\lambda_k$}
\begin{proposition}\label{prop:hatlam_lt_1}
    For all $k\ge2$, $\hat\lambda_k\in[1/\kappa,r_k]$.
\end{proposition}
\begin{proof}
{\em Upper bound:}
By definition $h_k(k;r_k)=1$, which means that if $\lambda=r_k$ then the Poisson distribution of order $k$ does not have a unique mode of zero.
Also by definition $\hat\lambda_k$ is the smallest value of $\lambda$ such that a double mode is attained.
Hence $\hat\lambda_k \le r_k$.
\\
{\em Lower bound:}
Philippou \cite{PhilippouFibQ} proved that the Poisson distribution of order $k$ has a unique mode of zero if $\lambda < 1/\kappa$.
But if the distribution has a double mode then zero is not a unique mode.
Hence $\hat\lambda \ge 1/\kappa$.
\end{proof}
\begin{remark}
  Recall also that $\hat{m}_k \ge k$.
  Hence we can further note if $\hat{m}_k = k$, then $\hat\lambda_k=r_k$.
  Philippou \cite{PhilippouFibQ} proved that for $k=2$ the double modes are at $0$ and $2$.
  Hence the upper bound $\hat\lambda_k \le r_k$ is sharp.
  We can write the following sequence of inequalities
\bq
\label{eq:ineq_hatlam_rk}
1/\kappa \le \hat\lambda_k \le r_k < 1 \,.
\eq
\end{remark}

\subsection{Case $\hat{m}_k=k$}
Fix a value $k\ge2$ and by hypothesis $\hat{m}_k = k$.
We know such a case exists because Philippou \cite{PhilippouFibQ} proved that $\hat{m}_2 = 2$ for $k=2$.  
Then because $\hat{m}_{k+1} \ge k+1$, it is immediate that $\hat{m}_{k+1} > \hat{m}_k$.  
Next set $n=\hat{m}_k(=k)$ and invoke Prop.~\ref{prop:hk_same_k_ge_n} to deduce
$h_{k+1}(k;\lambda) = h_k(k;\lambda)$, for all $\lambda>0$.
Next set $\lambda = \hat\lambda_k$.
Then $h_k(k;\hat\lambda_k)=1$ hence
$h_{k+1}(k;\hat\lambda_k) = 1$.
However, Lemma 1 by Kwon and Philippou \cite{KwonPhilippou} then tells us that
$h_{k+1}(k+1;\hat\lambda_k) > h_{k+1}(k;\hat\lambda_k)$, hence
$h_{k+1}(k+1;\hat\lambda_k) > 1$.
Hence the value $\hat\lambda_{k+1}$ which yields the first double mode for $h_{k+1}$ is a strictly smaller number: $\hat\lambda_{k+1} < \hat\lambda_k$.
This establishes that if $\hat{m}_k=k$, then (i) $\hat{m}_{k+1}>\hat{m}_k$ and (ii) $\hat\lambda_{k+1} < \hat\lambda_k$.

\subsection{Case $\hat{m}_k>k$}
Fix a value $k \ge 2$ and by hypothesis $\hat{m}_k>k$.
Such a case is known to exist, e.g.~$k=15$ \cite{Mane_Poisson_k_CC23_3}.
Set $x=\hat{m}_k$ in eq.~\eqref{eq:hk_inc_seq_k_le_n} in Prop.~\ref{prop:hk_inc_seq_k_le_n} to deduce
$h_{k+1}(\hat{m}_k;\lambda) > h_k(\hat{m}_k;\lambda)$.
Next set $\lambda=\hat\lambda_k$, so $h_k(\hat{m}_k;\hat\lambda_k)=1$.
Then we deduce $h_{k+1}(\hat{m}_k;\hat\lambda_k) > 1$.
Next comes a key step:
to attain a double mode for $h_{k+1}(\cdot)$, we must have $h_{k+1}(n;\lambda) \le 1$ for {\em all} values $n > 0$.
The value $\hat\lambda_k$ is too large to be the smallest value of $\lambda$ to attain a double mode for $h_{k+1}(\cdot)$.
The value $\hat\lambda_{k+1}$ which yields the first double mode for $h_{k+1}(\cdot)$ is a strictly smaller number, i.e.~$\hat\lambda_{k+1} < \hat\lambda_k$.

To prove $\hat{m}_{k+1} > \hat{m}_k$ is more complicated and will be left for future work.
We sketch some of the complications.
Recall the histogram plot of $h_k(n;\lambda)$ for the Poisson distribution of order $50$ in Fig.~\ref{fig:hist50}.
It exhibits the first double mode at $0$ and $\hat{m}_k=113$ (and $\hat\lambda_k \simeq 0.10194$).
As already noted, the value of $h_k(n;\lambda)$ is increasing for $1 \le n \le k (=50)$
which is the increasing sequence of Lemma 1 by Kwon and Philippou \cite{KwonPhilippou} (see eq.~\eqref{eq:KP_Lemma1_inc_seq}).
However for $n>k$, which is what we need here, the structure of the histogram is complicated.
There is a local maximum at $n=98$ but its height is $0.9835\dots$ and is not a global maximum.
The (joint) global maximum is at $0$ and $113$.
It is difficult to establish an ``increasing sequence'' $\{h_k(n;\lambda),h_k(n+1;\lambda),\dots\}$ as in Lemma 1 by Kwon and Philippou \cite{KwonPhilippou}.
This makes it hard(er) to prove $\hat{m}_{k+1} > \hat{m}_k$ when $\hat{m}_k > k$.

\subsection{Graphs}
For the record, we plot graphs of the location of the first double mode $\hat{m}_k$ as a function of $k$.
Fig.~\ref{fig:mode_first_double_100} displays a plot of $\hat{m}_k$ for $2 \le k \le 100$.
Observe that the graph consists of linear segments, with jumps from one segment to another,
and the slopes of successive segments get steeper (the slope values are $1$, $2$ and $3$).
Basically, the value of $\hat{m}_k$ is nonlinear in $k$, but is also constrained to be an integer,
and this is how it demonstrates that fact.  
Fig.~\ref{fig:mode_first_double_10000} displays a plot of $\hat{m}_k$ for $2 \le k \le 10^4$.
Note that the horizontal axis plots the value of $\kappa$ (divided by $10^6$).
The dashed curve is a fit $2.34\,\kappa^{5/8}$.
It is an approximate asymptotic fit.

\newpage
\setcounter{equation}{0}
\section{\label{sec:notmode}Integers which cannot be values of the mode}
In the last line of their paper, Kwon and Philippou \cite{KwonPhilippou} posed the question: 
``What positive integers cannot be modes of the Poisson distribution of order $k$?''

First we dispose of the case $k=1$, the standard Poisson distribution.
It is well-known that all nonnegative integers $n\ge0$ can be modes of the Poisson distribution.
The question is nontrivial only for $k\ge2$.
Next, since we proved above that the value of the first double mode must be at least $k$ (recall $\hat{m}_k \ge k$),
a partial answer is that for $k\ge2$, integers in the interval $[1,k-1]$ cannot be modes of the Poisson distribution of order $k$.
In particular, the value $1$ is never a mode for any $k\ge2$.

In this section, we give an ``almost complete'' answer to the above question.
We say ``almost complete'' because the answer relies partly on numerical conjectures.
If the conjectures are proved, then the answer below is complete.
One numerical conjecture is this: for $k\ge2$ and $n\ge2\kappa$ and $\lambda=n/\kappa$,
the mode $m_k(\lambda)$ of the Poisson distribution of order $k$ is given by (eq.~(4.1) in \cite{Mane_Poisson_k_CC23_3})
\bq
m_k(n/\kappa) = n - \biggl\lfloor\frac{3k+5}{8}\biggr\rfloor \,.
\eq
The above conjecture implies that any integer 
$i \ge k(k+1) - \lfloor(3k+5)/8\rfloor$ can be a mode of the Poisson distribution of order $k$.
Hence, if the above conjecture is correct, then for any given $k\ge2$
it is only necessary to test integers in the interval $[k,k(k+1) - \lfloor(3k+5)/8\rfloor)$.

We employ the term ``excluded values'' for positive integers which cannot be modes of the Poisson distribution of order $k$.
The value of the mode was calculated numerically for $k\ge2$ and the following results were obtained.
\begin{enumerate}
\item
  For all $k\ge2$, integers in the interval $[1,\hat{m}_k-1]$ are excluded.
\item
  Because $\hat\lambda_k \le r_k$ (see Prop.~\ref{prop:hatlam_lt_1} and eq.~\eqref{eq:ineq_hatlam_rk}),
  Lemma 3 by Kwon and Philippou \cite{KwonPhilippou} tells us that $k+1$ can never be a mode, for any $k\ge2$.
  Numerical calculations confirm this.
\item
  For $2 \le k \le 14$, the set of excluded values is listed in Table~\ref{tb:excluded1}.
  Note that $\hat{m}_k=k$ for $k\in[2,14]$.
\item
  For $15 \le k \le 28$, the set of excluded values is listed in Table~\ref{tb:excluded2}.
  Note that $\hat{m}_k=2k-5$ for $k\in[15,28]$.
\item
  For $29 \le k \le 37$, the set of excluded values is listed in Table~\ref{tb:excluded3}.
  Note that $\hat{m}_k=2k-4$ for $k\in[29,37]$.
\item
  For $38 \le k \le 41$, the set of excluded values is listed in Table~\ref{tb:excluded4}.
  Note that $\hat{m}_k=2k-3$ for $k\in[38,41]$.
\item
  For $k\ge42$, the interval $[1,\hat{m}_k-1]$ is the {\em only} interval of excluded values.
  The number $0$ and all integers $\ge\hat{m}_k$ can be modes of the Poisson distribution of order $k$.
\end{enumerate}
{\em Note that the last statement ``$k\ge42$'' is itself a numerical conjecture.}
All values $42 \le k \le 100$ were tested, and Monte Carlo sampling for values $101 \le k \le 500$. No counterexamples were found.
The computations were too slow to go up to $k=1000$.
Hence if the numerical conjectures are proved, the above is a complete tabulation of the integers which cannot be modes of the Poisson distribution of order $k$.

\newpage
\section{\label{sec:conc}Conclusion}
One of the two foci of this note was the ``first double mode'' by which is meant the first occasion (smallest value of $\lambda$)
the Poisson distribution of order $k\ge2$ is bimodal, with modes at $0$ and $m>0$.
Denoting the corresponding value of $\lambda$ by $\hat\lambda_k$, it was shown that $\hat\lambda_k \in[1/\kappa,r_k]$
where $\kappa=k(k+1)/2$ and $r_k$ is the unique positive root of the equation $h_k(k;r_k)=1$.
It was also shown that the sequence $\{\hat\lambda_k,\,k=2,3,\dots\}$ is strictly decreasing in $k$.
(We can trivially include the mode value $m=1$ with corresponding value $\lambda=1$ for the case $k=1$.)
The corresponding value of $m$ for a double mode is denoted by $\hat{m}_k$.
It was shown that $\hat{m}_k \ge k$.
It was found in \cite{Mane_Poisson_k_CC23_3} that $\hat{m}_k = k$ for $2 \le k \le 14$, but $\hat{m}_k > k$ for $k \ge 15$.
Numerical calculations show that asymptotically $\hat{m}_k \simeq 2.34\,\kappa^{5/8}$ (see Fig.~\ref{fig:mode_first_double_10000}).

Next, Sec.~\ref{sec:notmode} presented an ``almost complete'' answer to the question (posed in \cite{KwonPhilippou})
``which positive integers cannot be modes of the Poisson distribution of order $k$?''
The answer is ``almost complete'' because some parts rely on numerical conjectures.
If the numerical conjectures are proved, the answer is complete.

{\em In closing, the calculations in this note do not rule out the possibility of a triple mode.}
It is possible that there exist integers $0<m_1<m_2$ such that $0$, $m_1$ and $m_2$ are joint modes.
In such a case, $\hat\lambda_k$ would have to be a simultaneous positive real root of the equations $h_k(m_1;\hat\lambda_k)=1$ and $h_k(m_2;\hat\lambda_k)=1$.
Such a possibility was not investigated in this note.
More generally, it is an open question if the Poisson distribution of order $k$ has triple (or quadruple, etc.) modes,
not necessarily including zero as one of the mode values.


\newpage
\section{\label{sec:corrigendum}Corrigendum 9/19/2023}
After submission of the first version of this post, it was realized that Fig.~\ref{fig:mode_first_double_10000} contains an inconsistency.
It was shown in \cite{Mane_Poisson_k_CC23_3} that the value of $\kappa\lambda$, i.e.~the mean, at the location of the first double mode
scales as $k^{9/8}$ (see Fig.~4 in \cite{Mane_Poisson_k_CC23_3}).
According to the fit in Fig.~\ref{fig:mode_first_double_10000}, the location of the first double mode $\hat{m}_k$ scales as $\kappa^{5/8}$.
Since $\kappa=k(k+1)/2$, this implies that the location of the first double mode $\hat{m}_k$ scales as $k^{10/8}$.
{\em However, this is impossible because it was proved in Theorem 2.1 in \cite{PhilippouGeorghiouPhilippou} that the value of the mode never exceeds the mean.}
Note also that Fig.~4 in \cite{Mane_Poisson_k_CC23_3} is a logarithmic plot (which is better suited to estimate exponents of growth rates),
whereas Fig.~\ref{fig:mode_first_double_10000} is not.
{\em The error of this inconsistency in the analysis is regretted.}
Fig.~\ref{fig:mean_mode_first_double_10000} displays a logarithmic plot of the mean (solid curve) and first double mode (dotdash curve) together.
It is seen that the mode is always less than the mean and that asymptotically, they have equal (power law) growth rates.
The dashed line is a fit ($\textrm{constant}\times k^{9/8}$), as in Fig.~4 in \cite{Mane_Poisson_k_CC23_3}.

It is more informative to study the {\em difference} between the mean $\mu$ and the first double mode $\hat{m}_k$.
Fig.~\ref{fig:mean_first_mode_diff_10000} displays a plot of the difference $\mu-\hat{m}_k$ for $2 \le k \le 10^4$ (dashed line)
for the Poisson distribution of order $k$.
The dotted line is a fit $3.0 + 0.38\,k$ and is visually indistinguishable from the data.
The above reults indicate the following.
\begin{enumerate}
\item
  The value of the mode is less than the mean (correct).
\item
  The difference between the mean and first double mode is linear in $k$.
\item
  The values of both the mean and the first double mode scale asymptotically as $k^{9/8}$.
\end{enumerate}
For fixed $k\ge1$ and $\lambda>0$, and denoting the mode by $m_{k,\lambda}$, Theorem 2.1 in \cite{PhilippouGeorghiouPhilippou} states that
\bq
\lfloor \lambda k(k+1)/2\rfloor - \frac{k(k+1)}{2} +1 -\delta_{k,1} \le m_{k,\lambda} \le \lfloor \lambda k(k+1)/2\rfloor \,.
\eq
Given the results in this updated note (and treating $k\ge2$ only to avoid the Kronecker delta),
the lower bound can be sharpened to subtract $O(k)$.
\begin{conjecture}
For fixed $k\ge2$ and $\lambda>0$, an improved lower bound for the mode $m_{k,\lambda}$ is
\bq
\max\{0,\,\lfloor \lambda k(k+1)/2\rfloor - (c_0 + c_1\,k)\} \le m_{k,\lambda} \,.
\eq
Here $c_0$ and $c_1$ are constants to be determined.
\end{conjecture}

\newpage
\begin{figure}[!htb]
\centering
\includegraphics[width=0.75\textwidth]{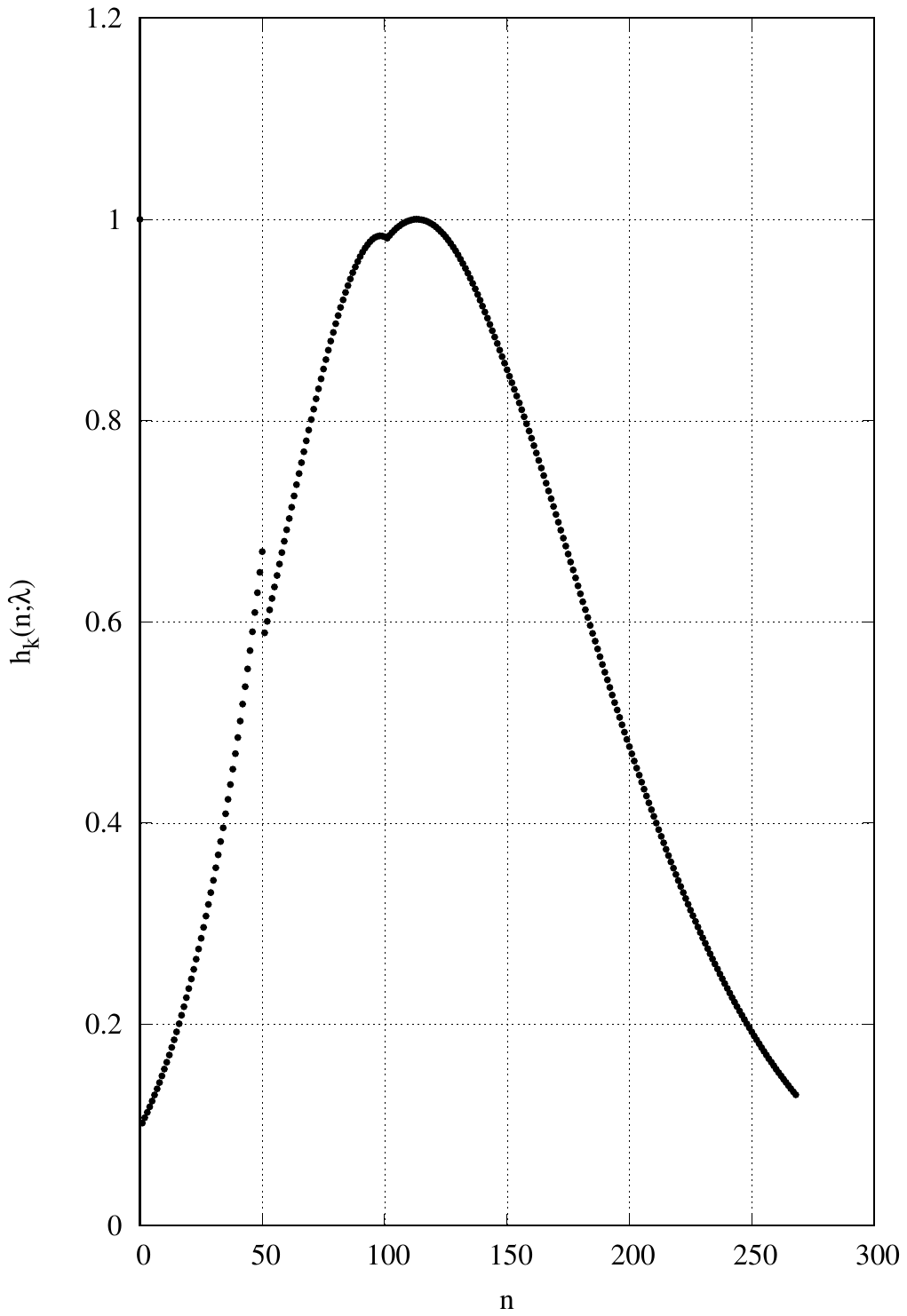}
\caption{\small
\label{fig:hist50}
Histogram plot of $h_k(n;\lambda)$ for the Poisson distribution of order $50$, exhibiting the first double mode (at $n=0$ and $113$).}
\end{figure}

\newpage
\begin{figure}[!htb]
\centering
\includegraphics[width=0.75\textwidth]{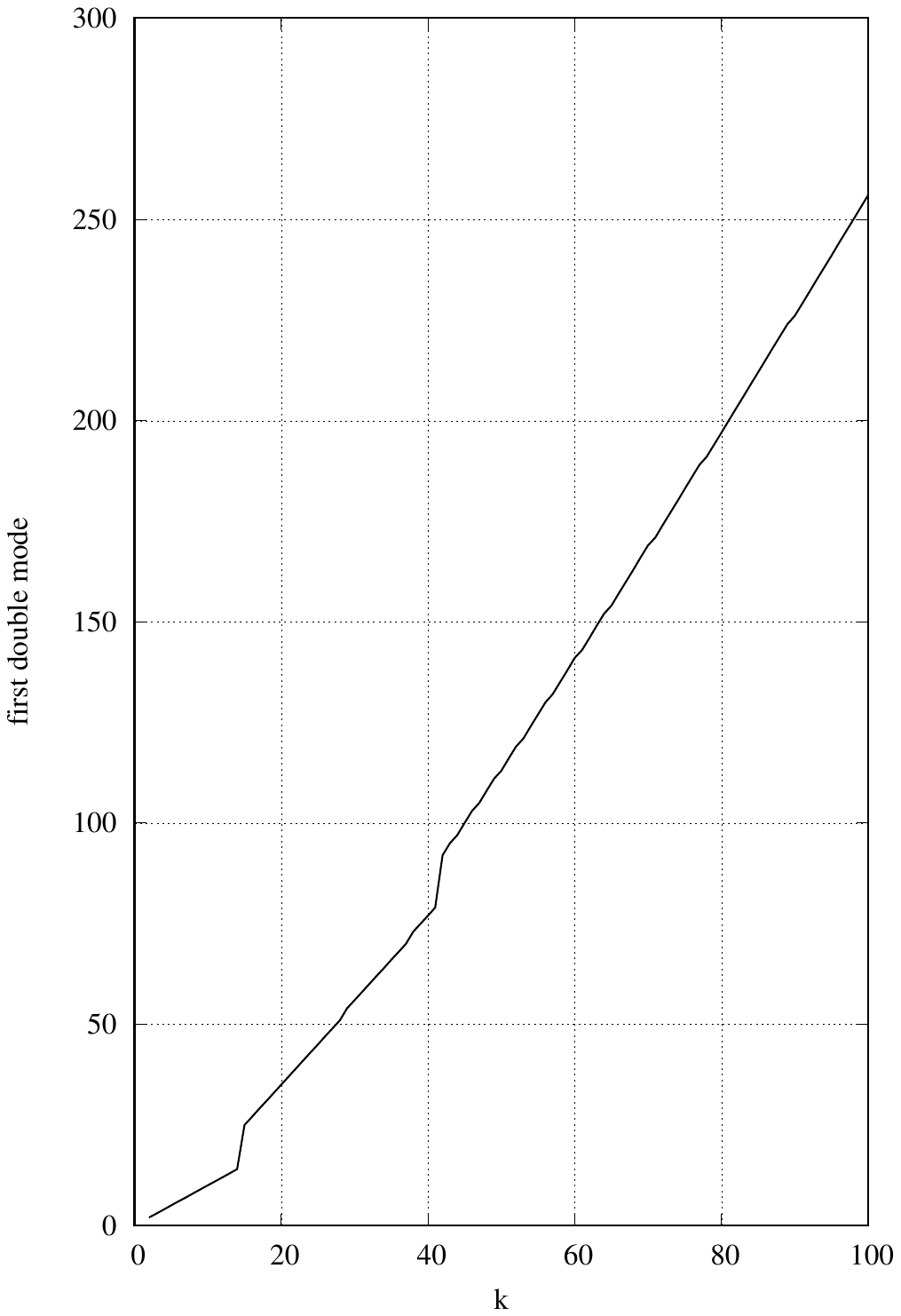}
\caption{\small
\label{fig:mode_first_double_100}
Plot of the location of the first double mode $\hat{m}_k$ for $2 \le k \le 100$.}
\end{figure}

\newpage
\begin{figure}[!htb]
\centering
\includegraphics[width=0.75\textwidth]{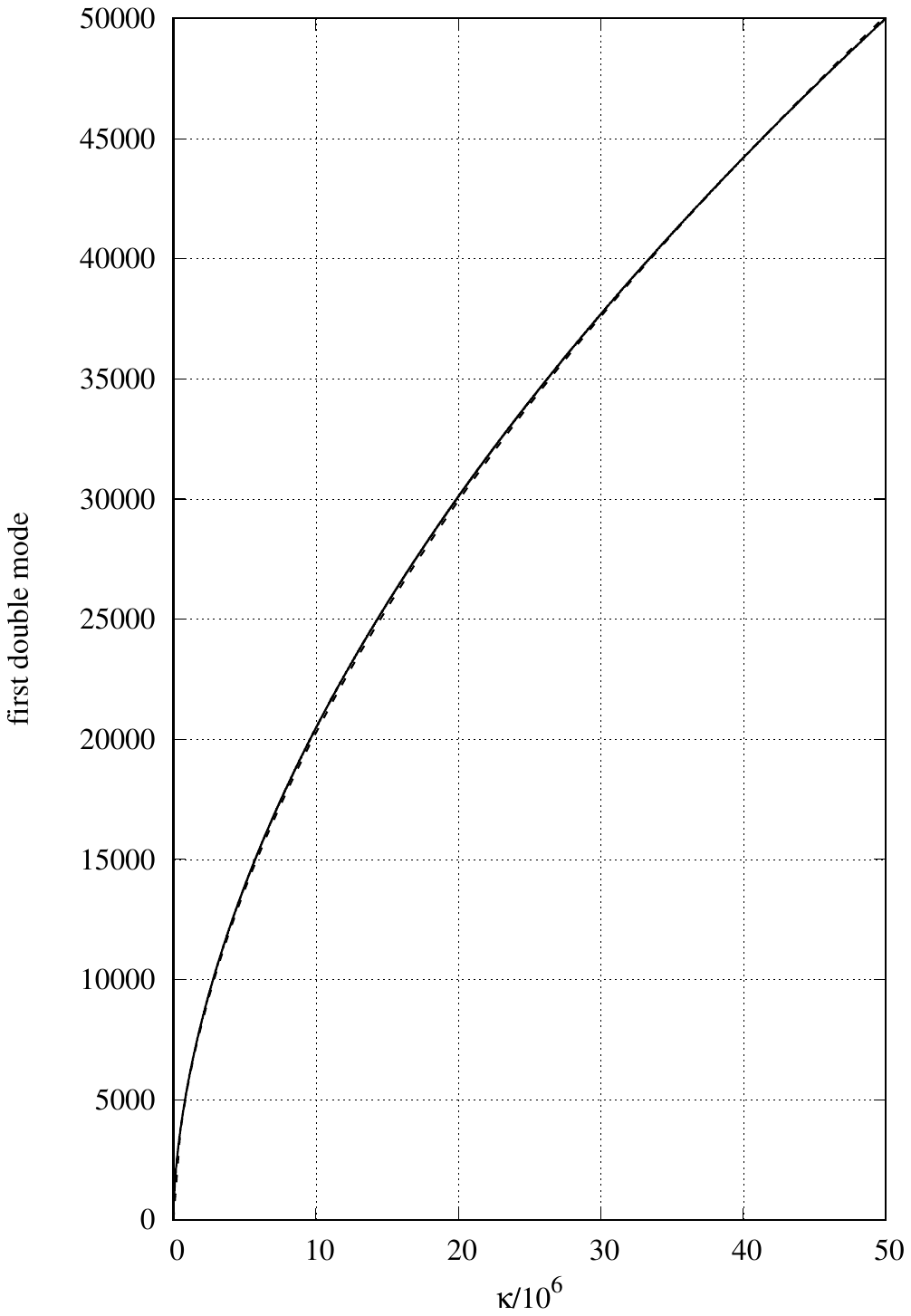}
\caption{\small
\label{fig:mode_first_double_10000}
Plot of the location of the first double mode $\hat{m}_k$ for $2 \le k \le 10^4$.
The horizontal axis shows the value of $\kappa=k(k+1)/2$, divided by $10^6$.
The dashed line is the fit $2.34\,\kappa^{5/8}$.}
\end{figure}

\newpage
\begin{figure}[!htb]
\centering
\includegraphics[width=0.75\textwidth]{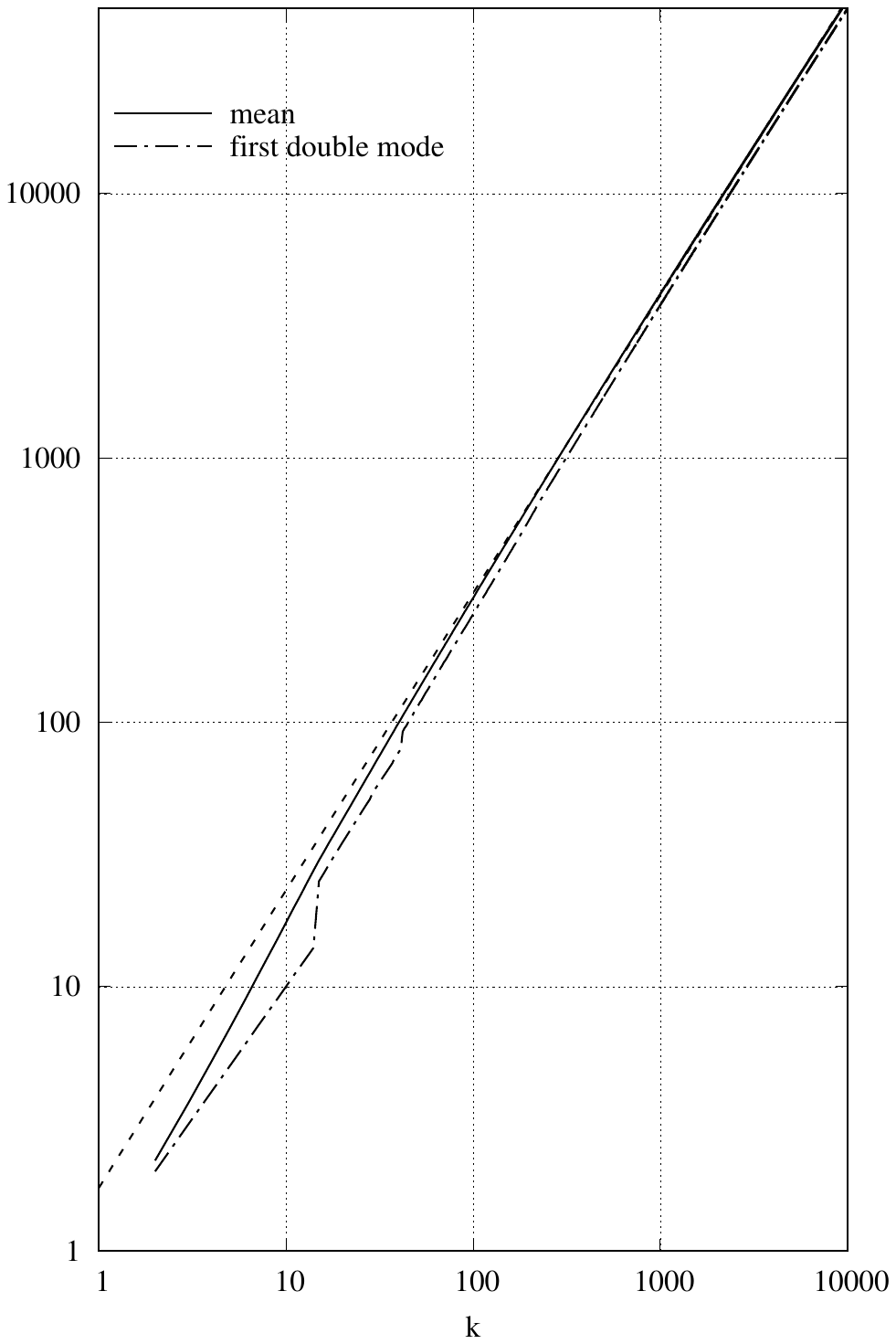}
\caption{\small
\label{fig:mean_mode_first_double_10000}
Logarithmic plot of the value of the mean (solid) and first double mode (dotdash) for the Poisson distribution of order $k$ for $2 \le k \le 10^4$.
The dashed line is a power law fit $\textrm{const}\times k^{1.125}$.}
\end{figure}

\newpage
\begin{figure}[!htb]
\centering
\includegraphics[width=0.75\textwidth]{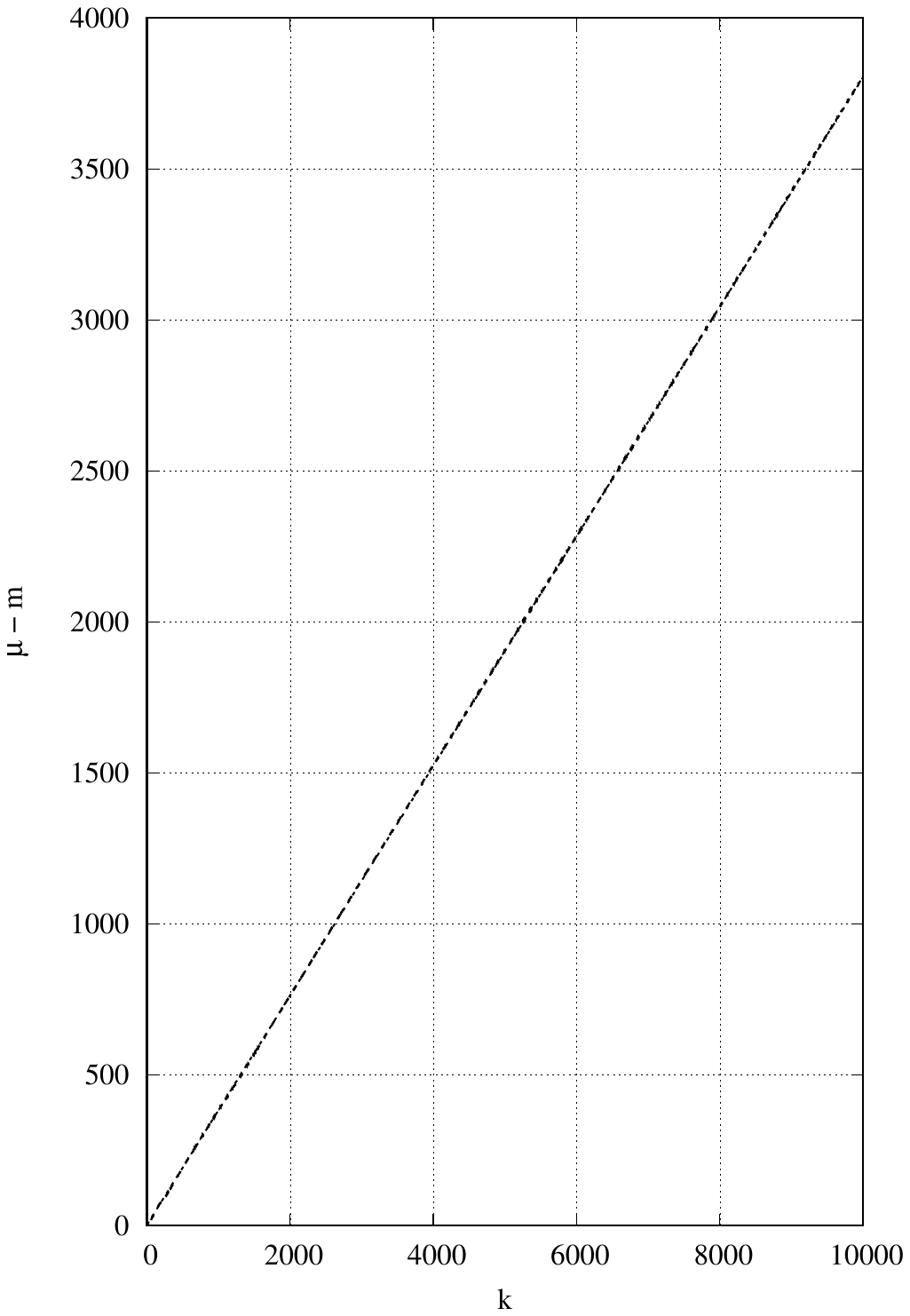}
\caption{\small
\label{fig:mean_first_mode_diff_10000}
Plot of the difference between the mean $\mu$ and the first double mode $\hat{m}_k$ for $2 \le k \le 10^4$ (dashed line)
for the Poisson distribution of order $k$.
The dotted line is a fit $3.0 + 0.38\,k$ (visually indistinguishable).}
\end{figure}

\newpage
\begin{table}[htb]
\centering
\begin{tabular}[width=0.75\textwidth]{|l|l|l|l|}
  \hline
  $k$ & $1^{st}$ interval & $2^{nd}$ interval & $3^{rd}$ interval \\
  \hline
  2 & $[1,k-1]$ & $3$ & \\
  3 & & $4$ & \\
  4 & & $[5,6]$ & 9 \\
  5 & & $[6,8]$ & 11 \\
  6 & & $[7,9]$ & 13 \\
  7 & & $[8,11]$ & 15 \\
  8 & & $[9,13]$ & 17 \\
  9 & & $[10,14]$ & [19,20] \\
  10 & & $[11,16]$ & [21,22] \\
  11 & & $[12,18]$ & [23,24] \\
  12 & & $[13,19]$ & [25,26] \\
  13 & & $[14,21]$ & [27,28] \\
  14 & & $[15,23]$ & [29,30] \\
  \hline
\end{tabular}
\caption{\label{tb:excluded1}
  Tabulation of ``excluded values'' (integers which cannot be modes of the Poisson distribution of order $k$) for $k=2$ through $14$.
  The interval is $[1,k-1]$ in all rows in the first column.}
\end{table}

\newpage
\begin{table}[htb]
\centering
\begin{tabular}[width=0.75\textwidth]{|l|l|l|}
  \hline
  $k$ & $1^{st}$ interval & $2^{nd}$ interval \\
  \hline
  15 & $[1,2k-6]$ & $[30,32]$ \\
  16 & & $[32,35]$ \\
  17 & & $[34,37]$ \\
  18 & & $[36,39]$ \\
  19 & & $[38,41]$ \\
  20 & & $[40,43]$ \\
  21 & & $[42,45]$ \\
  22 & & $[44,48]$ \\
  23 & & $[46,50]$ \\
  24 & & $[48,52]$ \\
  25 & & $[50,54]$ \\
  26 & & $[51,56]$ \\
  27 & & $[53,58]$ \\
  28 & & $[55,61]$ \\
  \hline
\end{tabular}
\caption{\label{tb:excluded2}
  Tabulation of ``excluded values'' (integers which cannot be modes of the Poisson distribution of order $k$) for $k=15$ through $28$.
  The interval is $[1,2k-6]$ in all rows in the first column.}
\end{table}

\newpage
\begin{table}[htb]
\centering
\begin{tabular}[width=0.75\textwidth]{|l|l|l|}
  \hline
  $k$ & $1^{st}$ interval & $2^{nd}$ interval \\
  \hline
  29 & $[1,2k-5]$ & $[57,63]$ \\
  30 & & $[59,65]$ \\
  31 & & $[61,67]$ \\
  32 & & $[63,69]$ \\
  33 & & $[65,71]$ \\
  34 & & $[67,73]$ \\
  35 & & $[69,76]$ \\
  36 & & $[70,78]$ \\
  37 & & $[72,80]$ \\
  \hline
\end{tabular}
\caption{\label{tb:excluded3}
  Tabulation of ``excluded values'' (integers which cannot be modes of the Poisson distribution of order $k$) for $k=29$ through $37$.
  The interval is $[1,2k-5]$ in all rows in the first column.}
\end{table}

\newpage
\begin{table}[htb]
\centering
\begin{tabular}[width=0.75\textwidth]{|l|l|l|}
  \hline
  $k$ & $1^{st}$ interval & $2^{nd}$ interval \\
  \hline
  38 & $[1,2k-4]$ & $[74,82]$ \\
  39 & & $[76,84]$ \\
  40 & & $[78,86]$ \\
  41 & & $[80,89]$ \\
  \hline
\end{tabular}
\caption{\label{tb:excluded4}
  Tabulation of ``excluded values'' (integers which cannot be modes of the Poisson distribution of order $k$) for $k=38$ through $41$.
  The interval is $[1,2k-4]$ in all rows in the first column.}
\end{table}

\end{document}